\documentclass[review]{elsarticle}
\usepackage{amsfonts,epsf,amsmath,amssymb,graphicx}
\usepackage{epstopdf}
\usepackage[top=2.5cm,bottom=2.5cm,left=2.5cm,right=2.5cm]{geometry}
\usepackage{lineno,hyperref}
\modulolinenumbers[5]

\newtheorem{theorem}{\bf Theorem}[section]
\newtheorem{corollary}[theorem]{\bf Corollary}

\journal{.}

\begin{document}

\begin{frontmatter}

\title{Relationship Between the Hosoya Polynomial and the Edge-Hosoya Polynomial of Trees}


\author[mymainaddress]{Niko Tratnik}
\ead{niko.tratnik1@um.si}
\author[mymainaddress,mysecondaryaddress]{Petra \v Zigert Pleter\v sek}
\ead{petra.zigert@um.si}

\address[mymainaddress]{Faculty of Natural Sciences and Mathematics, University of Maribor, Slovenia}
\address[mysecondaryaddress]{Faculty of Chemistry and Chemical Engineering, University of Maribor, Slovenia}

\begin{abstract}
We prove the relationship between the Hosoya polynomial and the edge-Hosoya polynomial of trees. The connection between the edge-hyper-Wiener index and the edge-Hosoya polynomial is established. With these results we also prove formulas for the computation of the edge-Wiener index and the edge-hyper-Wiener index of trees using the Wiener index and the hyper-Wiener index. Moreover, the closed formulas are derived for a family of chemical trees called regular dendrimers.
\end{abstract}


\end{frontmatter}

\linenumbers

\section{Introduction}
The first distance-based topological index was the Wiener index introduced in 1947 by H. Wiener \cite{Wiener}. Later, in 1988 H. Hosoya  \cite{hosoya} introduced some counting polynomials in chemistry, among them the Wiener polynomial, which is strongly connected to the Wiener index. Nowadays, it is known as the {\it Hosoya polynomial}. Another distance-based topological index, the hyper-Wiener index, was introduced in 1993 by M. Randi\' c \cite{randic}. All these concepts found many applications in different fields, such as chemistry, biology, networks.

The Hosoya polynomial, the Wiener index, and the hyper-Wiener index are based on the distances between pairs of vertices in a graph, and similar concepts have been introduced for distances between pairs of edges under the names the edge-Hosoya polynomial \cite{behm}, the edge-Wiener index \cite{iran-2009}, and the edge-hyper-Wiener index \cite{edge-hyper}. In this paper we study the relationships between the vertex-versions and the edge-versions of the Hosoya polynomial, the Wiener index, and the edge-Wiener index of trees.

\section{Preliminaries}

Unless stated otherwise, the graphs considered in this paper are connected. We define $d(x,y)$ to be the distance between vertices $u, v\in V(G)$. The distance $d(e,f)$ between edges $e$ and $f$ of graph $G$ is defined as the distance between vertices $e$ and $f$ in the line graph $L(G)$. 
\bigskip

\noindent
If $G$ is a connected graph with $n$ vertices, and if $d(G,k)$ 
is the number of (unordered) pairs of its vertices that are at distance $k$, 
then the \textit{Hosoya polynomial} of $G$ is defined as
$$
H(G,x) = \sum_{k \geq 0} d(G,k)\,x^k.
$$
Note that $d(G,0)= n$. Similarly, if $d_e(G,k)$ 
is the number of (unordered) pairs of edges that are at distance $k$, 
then the \textit{edge-Hosoya polynomial} of $G$ is defined as
$$
H_e(G,x) = \sum_{k \geq 0} d_e(G,k)\,x^k.
$$
Obviously, for any connected graph $G$ it holds $H_e(G,x) = H(L(G),x)$.
\bigskip

\noindent
The \textit{Wiener index} and the \textit{edge-Wiener index} of a connected graph $G$ are defined  in the following way:
$$W(G) = \sum_{\lbrace u,v\rbrace \subseteq V(G)}d(u,v), \ \ \quad W_e(G) = \sum_{\lbrace e,f\rbrace \subseteq E(G)}d(e,f).$$

\noindent
It is easy to see that $W_e(G) = W(L(G))$. The main property of the Hosoya polynomial and the edge-Hososya polynomial, that makes them interesting in chemistry, follows directly
from the definitions (see also \cite{petkovsek}):
\begin{equation} \label{reh} W(G) = H'(G,1), \quad \ \ W_e(G) = H_e'(G,1).
\end{equation}

\noindent
The \textit{hyper-Wiener index} and the \textit{edge-hyper-Wiener index} of a connected graph $G$ are defined as:
$$WW(G) = \frac{1}{2}\sum_{\lbrace u,v\rbrace \subseteq V(G)}d(u,v) + \frac{1}{2}\sum_{\lbrace u,v\rbrace \subseteq V(G)}d^2(u,v),$$
$$WW_e(G) = \frac{1}{2}\sum_{\lbrace e,f\rbrace \subseteq E(G)}d(e,f) + \frac{1}{2}\sum_{\lbrace e,f\rbrace \subseteq E(G)}d^2(e,f).$$

\noindent
Again, it holds $WW_e(G) = WW(L(G))$. Moreover, the following relationship was proved in \cite{cash} for any connected graph $G$:
\begin{equation} \label{ecash} WW(G) = H'(G,1) + \frac{1}{2}H''(G,1).
 \end{equation}

\section{The edge-Hosoya polynomial of trees}

In this section we first show how the edge-hyper-Wiener index of an arbitrary connected graph can be calculated from the edge-Hosoya polynomial. 

\begin{theorem}
\label{povezava}
Let $G$ be a connected graph. Then
$$WW_e(G) = H_e'(G,1) + \frac{1}{2}H_e''(G,1).$$
\end{theorem}



\begin{proof}
Using Equation \ref{ecash} we obtain
$$WW_e(G)=WW(L(G))=H'(L(G),1) + \frac{1}{2}H''(L(G),1)=H_e'(G,1) + \frac{1}{2}H_e''(G,1)$$
and the proof is complete. \qed
\end{proof}

\noindent
The following theorem is the main result of this paper.
\begin{theorem}
Let $T$ be a tree. Then
\begin{equation} \label{glavna} H_e(T,x) = \frac{1}{x}H(T,x) - \frac{|V(T)|}{x}
.\end{equation}.
\end{theorem}

\begin{proof}
It suffices to prove that
$$H(T,x) = xH_e(T,x) + |V(T)|.$$
Let $V_k$ be the set of all (unordered) pairs of vertices of $T$ that are at distance $k$ and let $E_k$ be the set of all (unordered) pairs of edges of $T$ that are at distance $k$, where $k \geq 0$. That means
$$V_k = \lbrace \lbrace x, y \rbrace \ | \ x,y \in V(T), \ d(x,y) = k \rbrace, $$ 
$$E_k = \lbrace \lbrace e, f \rbrace \ | \ e,f \in E(T), \ d(e,f) = k \rbrace. $$

We first show that for any $k \geq 1$ there exists a bijective function $F: V_k \rightarrow E_{k-1}$. To define $F$, let $k \geq 1$ and let $x,y \in V(T)$ such that $d(x,y) = k$. Furthermore, let $P$ be the unique path in $T$ connecting $x$ and $y$. Obviously, $d(x,y) = |E(P)|= k$. We define $e_x$ to be the edge of $P$ which has $x$ for one end-vertex. Similarly, $e_y$ is the edge of $P$ which has $y$ for one end-vertex. It is easy to see that $d(e_x,e_y) = k-1$. With this notation we can define
$$F (\lbrace x,y\rbrace ) = \lbrace e_x,e_y \rbrace$$
for every $\lbrace x, y \rbrace \in V_k$.
Obviously, $F$ is a well-defined function. 

To show that $F$ is injective, let $\lbrace x,y\rbrace, \lbrace a, b \rbrace \in V_k$, $k \geq 1$, and suppose $F(\lbrace x,y\rbrace) = F(\lbrace a,b \rbrace)$. It follows that $\lbrace e_x, e_y \rbrace = \lbrace e_a, e_b \rbrace$ and without loss of generality we can assume $e_x = e_a$ and $e_y = e_b$. If $x=a$, we also get $y=b$, since otherwise $e_y \neq e_b$. Therefore, $\lbrace x, y \rbrace = \lbrace a,b\rbrace$. If $x \neq a$, it follows that $x=b$ and $y=a$. Again, $\lbrace x, y \rbrace = \lbrace a,b\rbrace$ and we are done.

To show that $F$ is surjective, we take $\lbrace e,f\rbrace \in E_{k-1}$. Let $x$ be the end-vertex of $e$ and $y$ the end-vertex of $f$ such that $d(x,y) = d(e,f) +1 = k$. Obviously, $x$ and $y$ are uniquely defined. It is easy to see that $F(\lbrace x,y\rbrace) = \lbrace e,f\rbrace$.

We have shown that for every $k \geq 1$ it holds $d(T,k) = |V_k| = |E_{k-1}|=d_e(T,k-1)$. It is also obvious that $d(T,0)=|V(T)|$. Hence, polynomials $H(T,x)$ and $ xH_e(T,x) + |V(T)|$ have the same coefficients. Therefore, Equation \ref{glavna} it true and the proof is complete. \qed
 
\end{proof}

\noindent
As a corollary we can now express the edge-Wiener index and the edge-hyper-Wiener index of trees with the Wiener index and the hyper-Wiener index.

\begin{corollary}
\label{posledica}
Let $T$ be a tree. Then
$$W_e(T) = W(T) - { {|V(T)|} \choose {2}}$$
and
$$WW_e(T)= WW(T) - W(T).$$
\end{corollary}

\begin{proof}
First we notice that if $G$ is a graph, then
\begin{equation} \label{pomoc} H(G,1) = \sum_{k \geq 0} d(G,k) = {{|V(G)|} \choose {2}} + |V(G)|. 
\end{equation}

\noindent
After differentiating Equation \ref{glavna} we obtain
\begin{equation} \label{odvod1} H_e'(T,x) = \frac{H'(T,x)x - H(T,x) + |V(T)|}{x^2}
 \end{equation}
and
\begin{align}
H_e''(T,x) &
 =  \frac{H''(T,x)x^3 - 2H'(T,x)x^2 + 2H(T,x)x -2x|V(T)|}{x^4}. \label{odvod2} 
\end{align}

\noindent
Using Equation \ref{odvod1} and Equation \ref{pomoc} it follows,
\begin{eqnarray*}
W_e(T) &= & H_e'(T,1) \\
& = & H'(T,1) - H(T,1) + |V(T)| \\
& = & W(T) - { {|V(T)|} \choose {2}}.
\end{eqnarray*}
Finally, Theorem \ref{povezava}, Equation \ref{odvod1}, Equation \ref{odvod2}, and Equation \ref{ecash} imply

\begin{eqnarray*}
WW_e(T) &= & H_e'(T,1) + \frac{1}{2}H_e''(T,1) \\
& = & H'(T,1) - H(T,1) + |V(T)| \\
& + & \frac{1}{2}H''(T,1) - H'(T,1) + H(T,1) - |V(T)| \\
& = & WW(T) - W(T).
\end{eqnarray*}

\qed
\end{proof}

\section{The edge-Hosoya polynomial of dendrimers}

Dendrimers are highly regular trees, which are of interest to
chemists, since they represent repetitively branched molecules. In this section we compute the edge-Hosoya polynomial, the edge-Wiener index and the edge-hyper-Wiener index of regular dendrimers. 

In particular, $T_{k,d}$ stands for the $k$-th \textit{regular dendrimer} of degree $d$. For any $d \geq 3$, $T_{0,d}$ is the one-vertex graph and $T_{1,d}$ is the star with $d+1$ vertices. Then for any $k \geq 2$ and $d \geq 3$, the tree $T_{k,d}$ is obtained by attaching $d-1$ new vertices of degree one to the vertices of degree one of $T_{k-1,d}$. For an example see Figure \ref{dendrimer}. The parameter $k$ corresponds to what in dendrimer chemistry is called ``number of generations'' \cite{gut-1994}.

\begin{figure}[h!] 
\begin{center}
\includegraphics[scale=0.6]{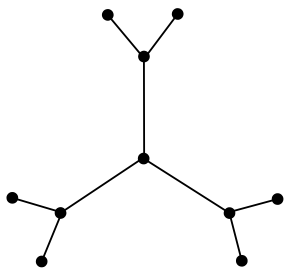}
\end{center}
\caption{\label{dendrimer} Regular dendrimer $T_{2,3}$.}
\end{figure} 

In \cite{sagan} the Wiener polynomial $W(G,x)$ of a graph $G$ was considered and the definition of this polynomial is slightly different from the definition of the Hosoya polynomial, such that $H(G,x) = W(G,x) + |V(G)|$. Hence, from Equation \ref{glavna} it follows
$$H_e(G,x) = \frac{1}{x}W(G,x).$$

\noindent
Therefore, to compute the edge-Hosoya polynomial we can use this formula and the result regarding the Wiener polynomial of a regular dendrimer in \cite{sagan}. After changing some labels we obtain
$$H_e(T_{k,d},x) = \sum_{i = 0}^{k-1}(d-1)^{2i}d\frac{(d-1)^{k-i}-1}{d-2}x^{2i} + \sum_{i=0}^{k-1}(d-1)^{2i}{{d} \choose {2}} \cdot \Bigg( d\frac{(d-1)^{k-i-1}-1}{d-2} + 1 \Bigg)x^{2i+1}.$$
It follows from Equation \ref{reh} and Theorem \ref{povezava} that the edge-Wiener index and the edge-hyper-Wiener index can be easily computed from the derivatives of the edge-Hosoya polynomial. Therefore, we obtain
$$W_e(T_{k,d})= \frac{d \Big(2-2 d+(d-1)^k (d^2+4d-4)+(d-1)^{2 k} (2-d (d+2)+2 (d-2) d k)\Big)}{2 (d-2)^3} $$
and

\begin{eqnarray*}
WW_e(T_{k,d}) &=  & d\frac{ 2 (d-1)+(d-1)^k \left(4-5 d^2\right)}{2
(d-2)^4} \\
&+ &  d\frac{ (d-1)^{2 k} \Big(-2-8 k+d \left(-2+5 d+16 k-d (d+4) k+2 (d-2)^2 k^2\right)\Big)}{2
(d-2)^4}.
\end{eqnarray*}
Since the Wiener index and the hyper-Wiener index of regular dendrimers are already known (see \cite{gut-1994,sagan,diudea}), the edge-Wiener index and the edge-hyper-Wiener index could also be computed in terms of Corollary \ref{posledica}.
\section*{Acknowledgment}
\noindent
Supported in part by the Ministry of Science of Slovenia under grant $P1-0297$.

\bibliography{mybibfile}

\end{document}